\documentclass{amsart}

\usepackage{amsfonts,amsmath,amssymb,amsthm,amscd,amsxtra}
\usepackage{enumerate,verbatim}

\usepackage[all,2cell,ps]{xy}
\usepackage[pagebackref]{hyperref}
\usepackage{mathptmx}

\theoremstyle{plain} 

\newtheorem{thm}{Theorem}[section]

\newtheorem{cor}[thm]{Corollary}

\theoremstyle{definition}
\newtheorem{lem}[thm]{Lemma}
\newtheorem{dfn}[thm]{Definition}

\newtheorem{eg}[thm]{Example}

\numberwithin{equation}{section}

\newcommand{\fp}{\mathfrak{p}}

\newcommand{\ZZ}{\mathbb{Z}}

\DeclareMathOperator{\G}{\mathcal{G}}

\DeclareMathOperator{\Ass}{Ass}

\DeclareMathOperator{\gd}{\mathcal{G}\textnormal{-dim}}

\DeclareMathOperator{\Tr}{Tr}
\DeclareMathOperator{\Tt}{\widehat{Tor}}
\DeclareMathOperator{\gt}{\mathcal{G}Tor}

\DeclareMathOperator{\coker}{coker}

\DeclareMathOperator{\depth}{depth}

\DeclareMathOperator{\Ext}{Ext}

\DeclareMathOperator{\hh}{H}
\DeclareMathOperator{\Hom}{Hom}

\DeclareMathOperator{\id}{id}

\DeclareMathOperator{\pd}{pd}

\DeclareMathOperator{\Spec}{Spec}

\DeclareMathOperator{\Tor}{Tor}

\def\urltilda{\kern -.15em\lower .7ex\hbox{\~{}}\kern .04em}
\def\urldot{\kern -.10em.\kern -.10em}\def\urlhttp{http\kern -.10em\lower -.1ex
\hbox{:}\kern -.12em\lower 0ex\hbox{/}\kern -.18em\lower 0ex\hbox{/}}

\begin{document}

\title[On the depth of tensor products of modules ]{On the depth of tensor products of modules}

\author[A. Sadeghi]{Arash Sadeghi}
\address{School of Mathematics, Institute for Research in Fundamental Sciences (IPM), P.O. Box: 19395-5746, Tehran, Iran} \email{sadeghiarash61@gmail.com}

\keywords{depth formula, $\G$-relative homology, Tate homology.\\
This research was supported by a grant from IPM}

\begin{abstract}
The depth of tensor product of modules over a Gorenstein local ring
is studied. For finitely generated nonzero modules $M$ and $N$ over a
Gorenstein local ring $R$, under some assumptions on the vanishing
of finite number of Tate and relative homology modules, the
$\depth_R(M\otimes_RN)$ is determined in terms of the $\depth_R(M)$
and $\depth_R(N)$.
\end{abstract}

\maketitle{}


\section{Introduction}
For finitely generated modules $M$ and $N$ over a commutative noetherian local ring $R$, the pair $(M,N)$ is said to satisfy the \emph{depth formula} provided
$$\depth R+\depth_R(M\otimes_RN)=\depth_R(M)+\depth_R(N).$$
The depth formula was first studied by Auslander. In \cite{A2}, he proved that the pair $(M,N)$ satisfies the depth formula provided that the projective dimension of $M$ is finite and $M$ and $N$ are Tor-independent, i.e., 
$\Tor^R_i(M,N)=0$ for all $i>0$.
Three decades later Huneke and Wiegand proved that the depth formula holds for 
Tor-independent modules over complete intersection local rings \cite{HW}.
Araya and Yoshino \cite{AY}, and Iyengar \cite{I} independently generalized Auslander's result for  modules of finite complete intersection dimension.
The finite complete intersection dimension condition was then relaxed by Christensen and Jorgensen \cite{CJ}. More precisely, they proved that the pair $(M,N)$ satisfies the depth formula provided that the Gorenstein dimension of $M$ is finite, $\Tor^R_i(M,N)=0$ for all $i>0$ and $\Tt_i^R(M,N)=0$ for all $i\in\ZZ$.
Recently, Celikbas, Liang and Sadeghi established that the depth formula holds under weaker assumptions \cite{CLS}:
\begin{thm}\label{CLS}
Let $M$ and $N$ be $R$--modules such that $M$ has finite Gorenstein dimension. Then
the depth formula holds provided the following conditions hold.
\begin{enumerate}[(i)]
\item{$\gt_i^R(M,N)=0$ for all $i>0$.}
\item{$\Tt_i^R(M,N)=0$ for all $i\leq0$.}
\end{enumerate}
\end{thm}
In this paper, our aim is to improve Theorem \ref{CLS} over Gorenstein local rings; rings over which each module has finite Gorenstein dimension. Let us remark that the condition (i) in Theorem \ref{CLS} is equivalent to the vanishing of $\gt_i^R(M,N)$ for all $i=1, \ldots, d$, where $d=\dim(R)$. However the condition (ii) cannot be obtained by the vanishing of finitely many Tate homology modules $\Tt_i^R(M,N)$ in general. The main feature of our result is that we require the vanishing of only finitely many Tate homology modules for the depth formula to hold. More precisely, we prove the following result; (see also Corollary \ref{c1} for a more general case).
\begin{thm}\label{MT}
Let $R$ be a Gorenstein local ring of dimension $d$ and let $M$ and $N$ be $R$--modules.
Assume that the following conditions hold:
\begin{enumerate}[(i)]
\item{$\gt_i^R(M,N)=0$ for all $i=1, \ldots, d$.}
\item{$\Tt_i^R(M,N)=0$ for all $i=-d, \ldots, 0$.}
\end{enumerate}
Then $\depth_R(M\otimes_RN)+\depth R=\depth_R(M)+\depth_R(N)$, i.e., the pair $(M,N)$ satisfies the depth formula.
\end{thm}
The organization of the paper is as follows. In Section 2, we collect preliminary notions,
definitions and some known results which will be used in this paper. In Section 3, we study the depth of tensor products of modules over Cohen-Macaulay local rings.
As a consequence, in Corollary \ref{c3}, we obtain a bound for the depth of tensor products of modules under some conditions weaker than those in Theorem \ref{MT}.
In Section 4, we prove our main result.
\section{Notations and preliminary results}
Throughout the paper, $R$ denotes a commutative noetherian local ring and all
$R$--modules are tacitly assumed to be finitely generated. 

The notion of the Gorenstein dimension was introduced by Auslander \cite{A1}, and developed by Auslander and Bridger in \cite{AB}.
\begin{dfn}
An $R$--module $M$ is said to be of Gorenstein dimension zero (or totally reflexive) whenever
the canonical map $M\rightarrow M^{**}$ is an isomorphism and $\Ext^i_R(M,R)=0=\Ext^i_R(M^*,R)=0$ for all $i>0$.
\end{dfn}
The Gorenstein dimension of $M$, denoted $\gd_R(M)$, is defined to be the infimum of all
nonnegative integers $n$ such that there exists an exact sequence
$$0\rightarrow G_n\rightarrow\cdots\rightarrow G_0\rightarrow  M \rightarrow 0,$$
where each $G_i$ are totally reflexive.

For a finite presentation $P_1\overset{f}{\rightarrow}P_0\rightarrow
M\rightarrow 0$ of an $R$--module $M$, its transpose, $\Tr M$, is
defined as $\coker f^*$, where $(-)^* := \Hom_R(-,R)$. Therefore we have the
following exact sequence:
\begin{equation}\label{1.1}
0\rightarrow M^*\rightarrow P_0^*\rightarrow P_1^*\rightarrow \Tr
M\rightarrow 0.
\end{equation}
Note that $\Tr M$ is unique up to projective equivalence.

In the following, we summarize some basic facts about Gorenstein dimension; see, for example, \cite{AB} for more details.
\begin{thm}\label{G}
For an $R$--module $M$, the following statements hold:
\begin{enumerate}[(i)]
       \item{$\gd_R(M)=0$ if and only if $\gd_R(\Tr M)=0$;}
       \item{(Auslander-Bridger formula) If $M$ has finite Gorenstein dimension, then $$\gd_R(M)=\depth R-\depth_R(M).$$}
        \item{$R$ is Gorenstein if and only if $\gd_R(M)<\infty$ for all finitely generated $R$--module $M$.}
\end{enumerate}
\end{thm}

Tate (co)homology for modules of finite Gorenstein dimension was studied by Avramov and Martsinkovsky in \cite{AvM}.
A complex $\mathbf{T}$ of free $R$--modules is called \emph{totally acyclic} if  $\hh_n(\mathbf{T})=0=\hh_n(\Hom_R(\mathbf{T},R))$ for all $n\in\ZZ$.
A \emph{complete resolution} of an $R$-module $M$ is a diagram
$$ \mathbf{T} \overset{\vartheta}{\longrightarrow} \mathbf{P} \overset{\pi}{\longrightarrow}M,$$
where $\pi$ is a projective resolution, $T$ is a totally
acyclic complex, $\vartheta$ is a morphism, and $\vartheta_i$ is an isomorphism
for $i \gg 0$. An $R$-module has finite Gorenstein dimension if and only
if it has a complete resolution.

Let $M$ be an $R$-module with a complete resolution $\mathbf{T}\rightarrow \mathbf{P}\rightarrow M$. For an $R$-module $N$, \emph{Tate homology} of $M$ and $N$ is defined as
$$\widehat{\Tor}_i^R(M,N) = \hh_i(\mathbf{T}\otimes_RN)  \text{ for } i\in\ZZ.$$
Also, for an $R$-module $N$, \emph{Tate cohomology} of $M$ and $N$ is defined as
$$ \widehat{\Ext}^i_R(M,N) = \hh^i(\Hom_R(\mathbf{T},N))  \text{ for } i\in\ZZ.$$
By construction, there are isomorphisms
\begin{equation}\label{t}
\widehat{\Tor}_i^R(M,N)\cong\Tor_i^R(M,N) \text{ and } \widehat{\Ext}^i_R(M,N)\cong\Ext^i_R(M,N),
\end{equation}
for all $i>\gd_R(M)$.
In the following we summarize some basic properties about Tate homology which will be used throughout the paper; see \cite{AvM} and \cite{CJ1} for more details.
\begin{thm}\label{tate}
Let $M$ and $N$ be $R$--modules such that $M$ has finite Gorenstein dimension. Then the following statements hold:
\begin{enumerate}[(i)]
\item{If either $\pd_R(M)<\infty$ or $\pd_R(N)<\infty$, then $\Tt_i^R(M,N)=0$ for all $i\in\ZZ$.}
\item{If $\gd_R(N)<\infty$, then $\Tt_i^R(M,N)\cong\Tt_i^R(N,M)$ for all $i\in\ZZ$.}
\item{If $\gd_R(M)=0$, then $\widehat{\Tor}_i^R(M,N)\cong\widehat{\Ext}^{-i-1}_R(M^*,N)$ for all $i\in\mathbb{Z}$.}
\item{If $N$ has finite injective dimension, then $\Tt_i^R(M,N)=0$ for all $i\in\ZZ$.}
\end{enumerate}
\end{thm}

A sequence $\eta$ of $R$-modules is called \emph{$\G$-proper} if the
induced sequence $\Hom_R(M,\eta)$ is exact for every totally
 reflexive $R$-module $M$.

A \emph{$\G$-proper resolution} of an $R$-module $M$ is a resolution
$\mathbf{L}\rightarrow M\rightarrow0$ by totally reflexive $R$-modules such that the augmented
resolution $\mathbf{L}^+$ is a $\G$-proper sequence. Every module of finite Gorenstein
dimension has a $\G$-proper resolution.

Let $M$ be an $R$-module with a $\G$-proper resolution $\mathbf{L}\rightarrow M\rightarrow0$. The
  \emph{$\G$-relative} homology of $M$ and $N$ is defined as
    $$\gt_{i}^R(M,N):= \hh_i(\mathbf{L}\otimes_RN) \ \text{for } i\in\ZZ.$$
It follows from the definition that $\gt_i^R(M,N)=0$ for all $i>\gd_R(M)$.

Next we summarize some basic properties about relative homology which will be used throughout the paper; see \cite{AvM} and \cite{Ia} for more details.
\begin{thm}\label{rel}
Let $M$ and $N$ be $R$--modules such that $M$ has finite Gorenstein dimension. The following statements hold:
\begin{enumerate}[(i)]
\item{If either $\pd_R(M)$ or $\pd_R(N)$ is finite, then
$$\gt_i^R(M,N)=\Tor_i^R(M,N) \text{ for all } i\geq0.$$}
\item{There is an exact sequence:
$$ \cdots\rightarrow \gt_{2}^R(M,N)\rightarrow\Tt_1^R(M,N)\rightarrow
  \Tor_{1}^R(M,N)\rightarrow\gt_{1}^R(M,N)\rightarrow0.$$}
\item{If $\gd_R(N)<\infty$, then $\gt_i^R(M,N)\cong\gt_i^R(N,M)$ for all $i\geq0$.}
\item{If $N$ has finite injective dimension, then $$\gt_i^R(M,N)=\Tor_i^R(M,N) \text{ for all } i\geq0.$$ }
\end{enumerate}
\end{thm}

The following results will be used throughout the paper.
\begin{lem}\label{l1}
Let $M$ and $N$ be $R$-modules such that $M$ has finite Gorenstein dimension. Assume that either $\pd_R(M)$ or $\id_R(N)$ is finite.  If $\Tor_i^R(M,N)=0$ for all $i>0$, then
$\depth R+\depth_R(M\otimes_RN)=\depth_R(M)+\depth_R(N).$
\end{lem}
\begin{proof}
See \cite[Theorem 1.2]{A2} and \cite[Theorem 2.13]{Y}.
\end{proof}
\begin{lem}\cite[Corollary 3.18]{CLS}\label{l3}
Let $M$ and $N$ be $R$--modules such that $M$ has finite Gorenstein dimension. Then
$$\sup\{i\mid\gt_i^R(M,N)\neq0\}\leq\sup\{\depth R_\fp-\depth_{R_\fp}(M_\fp)-\depth_{R_\fp}(N_\fp)\mid \fp\in\Spec R\}.$$
\end{lem}
\section{Depth formula and vanishing of Tate homology}
The following lemma plays a crucial role in this paper. Before giving a proof we recall that, for $R$--modules $M$ and $N$, the following sequence is exact:
\begin{equation}\label{1.2}
0\rightarrow \Ext^1_R(\Tr M,N)\rightarrow
M\otimes_RN\overset{\eta}{\rightarrow}\Hom_R(M^*,N)
\rightarrow\Ext^2_R(\Tr M,N)\rightarrow0,
\end{equation}
where $\eta$ is the evaluation map \cite[Proposition 2.6]{AB}.
\begin{lem}\label{l2}
Let $M$ be a totally reflexive $R$--module and let $N$ be an $R$--module. Assume $n$ is a nonnegative integer such that
$\Tt_i^R(M,N)=0$ for $-n\leq i\leq0$. Then $\depth_R(M\otimes_RN)\geq\min\{n+1,\depth_R(N)\}$.
\end{lem}
\begin{proof}
As $M^*=\Omega^2\Tr M$, by Theorem \ref{tate}(iii), Theorem \ref{G}(i) and (\ref{t}), we have the following isomorphisms
\[\begin{array}{rl}\tag{\ref{l2}.1}
\Tt_i^R(M,N)&\cong\widehat{\Ext}^{-i-1}_R(M^*,N)\\
&\cong\widehat{\Ext}^{-i+1}_R(\Tr M,N)\\
&\cong\Ext^{-i+1}_R(\Tr M,N)=0,
\end{array}\]
for all $-n\leq i\leq0$. If $n=0$ then we get the following exact sequence: $M\otimes_RN\hookrightarrow\Hom_R(M^*,N)$ by (\ref{l2}.1) and the exact sequence (\ref{1.2}).
Therefore, $\Ass_R(M\otimes_RN)\subseteq\Ass_R(N)$ and so the assertion is clear. If $n\geq1$ then by (\ref{l2}.1) and the exact sequence (\ref{1.2}) we get the following isomorphism,
\begin{equation}\tag{\ref{l2}.2}
M\otimes_RN\cong\Hom_R(M^*,N).
\end{equation}
Hence, $\depth_R(M\otimes_RN)=\depth_R(\Hom_R(M^*,N))\geq\min\{2,\depth_R(N)\}$ and so the assertion is clear for $n=1$. Now let $n>1$ and let
$\mathbf{P}\rightarrow M^*\rightarrow0$ be a free resolution of $M^*$.
Note that $\Ext^i_R(M^*,N)=0$ for $1\leq i\leq n-1$ by (\ref{l2}.1). Hence, by applying the functor $\Hom_R(-,N)$ to the free resolution of $M^*$, we get the
following exact sequence:
\begin{equation}\tag{\ref{l2}.3}
0\rightarrow\Hom_R(M^*,N)\rightarrow\Hom_R(P_0,N)\rightarrow\cdots\rightarrow\Hom_R(P_{n},N).
\end{equation}
Note that $\depth_R(\Hom_R(P_i,N))=\depth_R(N)$ for all $i\geq0$. It follows easily from the exact sequence (\ref{l2}.3) that
\begin{equation}\tag{\ref{l2}.4}
\depth_R(\Hom_R(M^*,N))\geq\min\{n+1,\depth_R(N)\}.
\end{equation}
Now the assertion follows from (\ref{l2}.2) and (\ref{l2}.4).
\end{proof}
\begin{thm}\label{t1}
Let $R$ be a Cohen-Macaulay local ring of dimension $d$ and let $M$ be an $R$--module of finite Gorenstein dimension and of depth $n$. Assume that $N$ is an $R$--module and that the following conditions hold:
\begin{enumerate}[(i)]
\item{$\Tt_i^R(M,N)=0$ for $-n<i\leq 0$.}
\item{$\gt_i^R(M,N)=0$ for $1\leq i\leq d-n$ (equivalently, for all $i>0$).}
\end{enumerate}
Then either $\depth_R(M\otimes_RN)+\depth R=\depth_R(M)+\depth_R(N)$ or \\ $\depth_R(M\otimes_RN)\geq\depth_R(M)$ .
\end{thm}
\begin{proof}
Assume that $\depth_R(M\otimes_RN)<\depth_R(M)$. We must prove that the pair $(M,N)$ satisfies the depth formula.
Without loss of generality, we may assume that $R$ is Cohen-Macaulay ring with canonical module.
Therefore, by \cite[Theorem A]{AuBu}, $N$ has a maximal Cohen-Macaulay approximation:
\begin{equation}\tag{\ref{t1}.1}
0\longrightarrow Q\longrightarrow Y\longrightarrow N\longrightarrow0,
\end{equation}
Where $Y$ is maximal Cohen-Macaulay and $Q$ has finite injective dimension. The exact sequence (\ref{t1}.1) induces the following long exact sequence
\begin{equation}\tag{\ref{t1}.2}
\cdots\rightarrow\Tt_i^R(M,Q)\rightarrow\Tt_i^R(M,Y)\rightarrow\Tt_i^R(M,N)\rightarrow\cdots,
\end{equation}
of Tate homology modules by \cite[Proposition 2.8]{CJ1}. As $\id_R(Q)<\infty$, $\Tt_i^R(M,Q)=0$ for all $i\in\ZZ$ by Theorem \ref{tate}(iv). Therefore, $\Tt_i^R(M,Y)\cong\Tt_i^R(M,N)$ for all $i\in\ZZ$ by (\ref{t1}.2).
In particular, by (i)
\begin{equation}\tag{\ref{t1}.3}
\Tt_i^R(M,Y)=0 \text{ for } -n<i\leq 0.
\end{equation}
As $\id_R(Q)<\infty$, the exact sequence (\ref{t1}.1) is proper and it induces the following long exact sequence
\[\begin{array}{rl}\tag{\ref{t1}.4}
\cdots\rightarrow\gt_1^R(M,Q)\rightarrow\gt_1^R(M,Y)\rightarrow\gt_1^R(M,N)\rightarrow\\
M\otimes_RQ\rightarrow M\otimes_R Y\rightarrow M\otimes_RN\rightarrow0,
\end{array}\]
of relative homology modules by the homology version of \cite[Proposition 4.4]{AvM} (see also \cite[Remark 7.4]{AvM}). By (ii), $\gt_1^R(M,N)=0$ and so we have the following exact sequence
\begin{equation}\tag{\ref{t1}.5}
0\rightarrow M\otimes_RQ\rightarrow M\otimes_R Y\rightarrow M\otimes_RN\rightarrow0,
\end{equation}
by (\ref{t1}.4). As $Y$ is maximal Cohen-Macaulay,
\begin{equation}\tag{\ref{t1}.6}
\gt_i^R(M,Y)=0 \text{ for all } i>0,
\end{equation}
by Lemma \ref{l3}.
It follows from (\ref{t1}.4), (\ref{t1}.6) and (ii) that $\gt_i^R(M,Q)=0$ for all $i>0$. Therefore, by Lemma \ref{l1} and Theorem \ref{rel}(iv)
\begin{equation}\tag{\ref{t1}.7}
\depth R+\depth_R(M\otimes_RQ)=\depth_R(M)+\depth_R(Q).
\end{equation}
Now consider the following exact sequence
\begin{equation}\tag{\ref{t1}.8}
0\longrightarrow X\longrightarrow G\longrightarrow M\longrightarrow0,
\end{equation}
where $\pd_R(X)<\infty$ and $\gd_R(G)=0$. As $X$ has a finite projective dimension, $\Tt_i^R(X,Y)=0$ for all $i\in\ZZ$ by Theorem \ref{tate}(i). Therefore, the exact sequence (\ref{t1}.8) induces the following isomorphism $\Tt_i^R(M,Y)\cong\Tt_i^R(G,Y)$ for all $i\in\ZZ$ by \cite[Proposition 2.9]{CJ1}. In particular, by (\ref{t1}.3) we have
\begin{equation}\tag{\ref{t1}.9}
\Tt_i^R(G,Y)=0 \text{ for } -n< i\leq 0.
\end{equation}
As $Y$ is maximal Cohen-Macaulay, it follows from (\ref{t1}.9) and Lemma \ref{l2} that
\begin{equation}\tag{\ref{t1}.10}
\depth_R(G\otimes_RY)\geq n.
\end{equation}
As $\pd_R(X)<\infty$ and $Y$ is maximal Cohen-Macaulay, $\Tor_i^R(X,Y)=0$ for all $i>0$ by \cite[Lemma 2.2]{Yo} and so
\begin{equation}\tag{\ref{t1}.11}
\depth_R(X\otimes_RY)=\depth_R(X)=n+1.
\end{equation}
by Lemma \ref{l1}. As $\pd_R(X)<\infty$, the exact sequence (\ref{t1}.8) is proper
and it induces the following long exact sequence
\begin{equation}\tag{\ref{t1}.12}
0\rightarrow X\otimes_RY\rightarrow G\otimes_R Y\rightarrow M\otimes_RY\rightarrow0.
\end{equation}
by the homology version of \cite[Proposition 4.6]{AvM} and (\ref{t1}.6).
It follows easily from (\ref{t1}.10), (\ref{t1}.11) and the exact sequence (\ref{t1},12) that
\begin{equation}\tag{\ref{t1}.13}
\depth_R(M\otimes_RY)\geq n.
\end{equation}
As $\depth_R(M\otimes_RN)<n$, by (\ref{t1}.13) and the exact sequence (\ref{t1}.5)
\begin{equation}\tag{\ref{t1}.14}
\depth_R(M\otimes_RN)=\depth_R(M\otimes_RQ)-1.
\end{equation}
Note that $\depth_R(N)=\depth_R(Q)-1$. Therefore the assertion follows from (\ref{t1}.14) and (\ref{t1}.7).
\end{proof}
As a consequence we have the following result.
\begin{cor}
Let $R$ be a Cohen-Macaulay local ring of dimension $d$ and let $M$, $N$ be $R$--modules. Assume that $M$ is totally reflexive and that
$\Tt_i^R(M,N)=0$ for $-d<i\leq 0$. Then either $M\otimes_RN$ is maximal Cohen-Macaulay or $\depth_R(M\otimes_RN)=\depth_R(N)$.
\end{cor}
\begin{proof}
First note that $\gt_i^R(M,N)=0$ for all $i>0$, because $M$ is totally reflexive. Therefore, by Theorem \ref{t1}, either the pair $(M,N)$ satisfies depth formula or
$\depth_R(M\otimes_RN)\geq\depth_R(M)$. Now the assertion follows from Auslander-Bridger formula.
\end{proof}
\begin{thm}\label{t5}
Let $R$ be a Cohen-Macaulay local ring of dimension $d$ and let $M$, $N$ be $R$--modules. Assume that $M$ has finite Gorenstein dimension and that the following conditions hold:
\begin{enumerate}[(i)]
\item{$\Tt_i^R(M,N)=0$ for $-d<i\leq 0$.}
\item{$\gt_i^R(M,N)=0$ for all $i>0$.}
\end{enumerate}
Then either $\depth_R(M\otimes_RN)+\depth R=\depth_R(M)+\depth_R(N)$ or \\ $\depth_R(M\otimes_RN)\geq\depth_R(N)$.
\end{thm}
\begin{proof}
Assume that $\depth_R(M\otimes_RN)<\depth_R(N)$. We must prove that the pair $(M,N)$ satisfies the depth formula.
Without loss of generality, we may assume that $R$ is Cohen-Macaulay ring with canonical module.
Consider the following exact sequence,
\begin{equation}\tag{\ref{t5}.1}
0\longrightarrow T\longrightarrow G\longrightarrow M\longrightarrow0,
\end{equation}
where $G$ is totally reflexive and $T$ has finite projective dimension. Therefore, $\Tt_i^R(T,N)=0$ for all $i\in\ZZ$ by Theorem \ref{tate}(i) and so
the exact sequence (\ref{t5}.1) induces the following isomorphism
\begin{equation}\tag{\ref{t5}.2}
\Tt_i^R(G,N)\cong\Tt_i^R(M,N) \text{ for all } i\in\ZZ,
\end{equation}
by \cite[Proposition 2.9]{CJ1}. Now consider the following exact sequence
\begin{equation}\tag{\ref{t5}.3}
0\longrightarrow N\longrightarrow I\longrightarrow L\longrightarrow0,
\end{equation}
where $L$ is maximal Cohen-Macaulay and $I$ has finite injective dimension (see \cite[Theorem A]{AuBu}).
By Theorem \ref{tate}(iv), $\Tt_i^R(G,I)=0$ for all $i\in\ZZ$. Hence, the exact sequence (\ref{t5}.3) induces the following isomorphism
\begin{equation}\tag{\ref{t5}.4}
\Tt_i^R(G,N)\cong\Tt_{i+1}^R(G,L) \text{ for all } i\in\ZZ,
\end{equation}
by \cite[Proposition 2.8]{CJ1}. It follows from (i), (\ref{t5}.2) and (\ref{t5}.4) that
\begin{equation}\tag{\ref{t5}.5}
\Tt_i^R(G,L)=0 \text{ for } -d+1<i\leq 1.
\end{equation}
As $G$ is totally reflexive, $\Tor_1^R(G,L)=0$ by (\ref{t}) and (\ref{t5}.5). Hence, the exact sequence (\ref{t5}.3) induces the following exact sequence:
\begin{equation}\tag{\ref{t5}.6}
0\rightarrow N\otimes_RG\rightarrow I\otimes_RG\rightarrow L\otimes_RG\rightarrow0.
\end{equation}
It follows from Lemma \ref{l2} and (\ref{t5}.5) that $\depth_R(L\otimes_RG)\geq d-1$. Therefore, by the exact sequence (\ref{t5}.6),
\begin{equation}\tag{\ref{t5}.7}
\depth_R(N\otimes_RG)\geq\min\{\depth_R(I\otimes_RG),d\}=\depth_R(I\otimes_RG).
\end{equation}
As $\id_R(I)<\infty$ and $G$ is totally reflexive, $\Tor_i^R(G,I)\cong\gt_i^R(G,I)=0$ for all $i>0$ by Theorem \ref{rel}(iv) and Lemma \ref{l3}.
Hence, by Lemma \ref{l1} the pair $(G,I)$ satisfies depth formula and so by Auslander-Bridger formula we have
\begin{equation}\tag{\ref{t5}.8}
\depth_R(I\otimes_RG)=\depth_R(I)=\depth_R(N).
\end{equation}
As $\pd_R(T)<\infty$, the exact sequence (\ref{t5}.1) is proper and it induces the following long exact sequence
\[\begin{array}{rl}\tag{\ref{t5}.9}
\cdots\rightarrow\gt_1^R(T,N)\rightarrow\gt_1^R(G,N)\rightarrow\gt_1^R(M,N)\rightarrow\\
T\otimes_RN\rightarrow G\otimes_R N\rightarrow M\otimes_RN\rightarrow0,
\end{array}\]
of relative homology modules. By (ii) and the long exact sequence (\ref{t5}.9) we obtain the following exact sequence
\begin{equation}\tag{\ref{t5}.10}
0\rightarrow T\otimes_RN\rightarrow G\otimes_R N\rightarrow M\otimes_RN\rightarrow0.
\end{equation}
As $G$ is totally reflexive, $\gt_i^R(G,N)=0$ for all $i>0$. Hence, it follows from (ii), Theorem \ref{rel}(i) and the exact sequence (\ref{t5}.9) that $\Tor_i^R(T,N)\cong\gt_i^R(T,N)=0$ for all $i>0$. Therefore, by Lemma \ref{l1} and Auslander-Buchsbaum formula
\[\begin{array}{rl} \tag{\ref{t5}.11}
\depth_R(T\otimes_RN)=\depth_R(N)-\pd_R(T)\\
=\depth_R(N)-\gd_R(M)+1.
\end{array}\]
As $\depth_R(M\otimes_RN)<\depth_R(N)$, it follows from (\ref{t5}.7), (\ref{t5}.8) and the exact sequence (\ref{t5}.10) that
\begin{equation}\tag{\ref{t5}.12}
\depth_R(T\otimes_RN)=\depth_R(M\otimes_R N)+1.
\end{equation}
Now the assertion follows from (\ref{t5}.11), (\ref{t5}.12) and Auslander-Bridger formula.
\end{proof}
The following is an immediate consequence of Theorem \ref{t1} and Theorem \ref{t5}.
\begin{cor}\label{c3}
Let $R$ be a Cohen-Macaulay local ring of dimension $d$ and let $M$, $N$ be $R$--modules. Assume that $M$ has finite Gorenstein dimension and that the following conditions hold:
\begin{enumerate}[(i)]
\item{$\Tt_i^R(M,N)=0$ for $-d<i\leq 0$.}
\item{$\gt_i^R(M,N)=0$ for all $i>0$.}
\end{enumerate}
Then at least one of the following holds:
\begin{enumerate}
\item{$\depth_R(M\otimes_RN)+\depth R=\depth_R(M)+\depth_R(N)$.}
\item{$\depth_R(M\otimes_RN)\geq\max\{\depth_R(M), \depth_R(N)\}$.}
\end{enumerate}
In particular, if either $M$ or $N$ is maximal Cohen-Macaulay then either $\depth_R(M\otimes_RN)=\depth_R(N)$ or $\depth_R(M\otimes_RN)=\depth_R(M)$.
\end{cor}
The following example shows that the depth formula may not hold under the assumptions of Corollary \ref{c3}.
In the next Section, we prove that the depth formula holds under some stronger conditions.
\begin{eg}\label{e1}\cite{HW1}
Let $R=k[[x,y]]/(xy)$ and let $M=R/(x)$, $N=R/(x^2)$. Then $R$ is a complete intersection ring with an isolated singularity of dimension one. Note that $M\otimes_RN\cong M$. Therefore, $M$ and $M\otimes_RN$ are maximal Cohen-Macaulay but $N$ is not maximal Cohen-Macaulay. In other words, the pair $(M,N)$ does not satisfy the depth formula. Now we prove that the assumptions of Corollary \ref{c3} hold. First note that $\gt_i^R(M,N)=0$ for all $i>0$, because $M$ is totally reflexive.
Consider the following exact sequence
\begin{equation}\tag{\ref{e1}.1}
0\rightarrow\Ext^1_R(\Tr M,N)\rightarrow M\otimes_RN\rightarrow\Hom_R(M^*,N)\rightarrow\Ext^2_R(\Tr M,N)\rightarrow0.
\end{equation}
Let $\fp\in\Ass_R(\Ext^1_R(\Tr M,N))$. As $M\otimes_RN$ is maximal Cohen-Macaulay, it follows from the exact sequence (\ref{e1}.1) that $\fp\in\Ass_R(M\otimes_RN)\subseteq\Ass R$ which is a contradiction, because $R$ has an isolated singularity. Therefore, by Theorem \ref{tate}(iii) and (\ref{t}), $\Tt_0^R(M,N)\cong\widehat{\Ext}^{-1}_R(M^*,N)\cong\Ext^1_R(\Tr M,N)=0$.
\end{eg}
\section{Main result}
In order to prove our main result, Theorem \ref{MT}, we will prove two Theorems which are of independent interest.
\begin{thm}\label{t2}
Let $R$ be a Cohen-Macaulay local ring of dimension $d$ and let $N$ be an $R$--module of depth $n$. Assume that $M$ is an $R$--module of finite and positive Gorenstein dimension and that the following conditions hold:
\begin{enumerate}[(i)]
\item{$\Tt_i^R(M,N)=0$ for $-n\leq i\leq 0$.}
\item{$\gt_i^R(M,N)=0$ for all $i>0$.}
\end{enumerate}
Then $\depth_R(M\otimes_RN)+\depth R=\depth_R(M)+\depth_R(N)$.
\end{thm}
\begin{proof}
Without loss of generality, we may assume that $R$ is Cohen-Macaulay ring with canonical module.
Consider the following exact sequence:
\begin{equation}\tag{\ref{t2}.1}
0\longrightarrow M\longrightarrow P\longrightarrow X\longrightarrow0,
\end{equation}
where $\pd_R(P)<\infty$ and $X$ is totally reflexive (see \cite[3.3]{CI}). As $P$ has finite projective dimension, $\Tt_i^R(P,N)=0$ for all $i\in\ZZ$ by Theorem \ref{tate}(i). Therefore, the exact sequence (\ref{t2}.1) induces the following isomorphism $\Tt_i^R(M,N)\cong\Tt_{i+1}^R(X,N)$ for all $i\in\ZZ$. In particular,
\begin{equation}\tag{\ref{t2}.2}
\Tt_i^R(X,N)=0 \text{ for } -n+1\leq i\leq 1,
\end{equation}
by (i). It follows from (\ref{t2}.2) and Lemma \ref{l2} that
\begin{equation}\tag{\ref{t2}.3}
\depth_R(X\otimes_RN)\geq\depth_R(N)=n.
\end{equation}
As $X$ is totally reflexive, by (\ref{t2}.2) and (\ref{t}),
\begin{equation}\tag{\ref{t2}.4}
\Tor_1^R(X,N)\cong\Tt_1^R(X,N)=0.
\end{equation}
Therefore, the exact sequence (\ref{t2}.1) induces the following exact sequence:
\begin{equation}\tag{\ref{t2}.5}
0\longrightarrow M\otimes_RN\longrightarrow P\otimes_RN\longrightarrow X\otimes_RN\longrightarrow0.
\end{equation}
Consider the exact sequence $0\rightarrow\Omega P\rightarrow F\rightarrow P\rightarrow0$, where $F$ is free. Taking the pull-back of the two maps into $P$, we get the following commutative diagram
$$\begin{CD}
&&&&&&&&\\
\ \ &&&& 0&&0\\
&&&& @VVV @VVV\\
&&&& \Omega P @>{=}>>\Omega P\\
&&&& @VVV @VVV\\
\ \ &&0@>>> Z@>>> F@>>> X@>>>0& \\
&&&& @VVV @VVV @V{\parallel}VV\\
\ \ &&0@>>> M @>>> P @>>> X @>>>0&  \\
&&&& @VVV  @VVV \\
\ \ &&&& 0&&0\\
\end{CD}$$
The usual properties of the pull-back and the Snake Lemma show that the rows and
columns of this diagram are exact.
Note that $Z$ is totally reflexive and so $\gt_i^R(Z,N)=0$ for all $i>0$.
As $\Omega P$ has finite projective dimension, the exact sequence $0\rightarrow\Omega P\rightarrow Z\rightarrow M\rightarrow0$ is proper. Therefore the above diagram
induces the following commutative diagram with exact rows
$$\begin{CD}
&&&&&&&&\\
\ \ &&0@>>>\gt_1^R(M,N)@>>>\Omega P\otimes_RN @>>> Z\otimes_RN& \\
&&&& @VVV @V{\parallel}VV @V{f}VV\\
\ \ &&0@>>>\Tor_1^R(P,N)@>>> \Omega P\otimes_RN  @>>> F\otimes_RN& \\
\end{CD}$$
It follows easily from (\ref{t2}.4) that $f$ is a monomorphism.
By (ii), $\gt_1^R(M,N)=0$. It follows from the above commutative diagram that
\begin{equation}\tag{\ref{t2}.6}
\Tor_1^R(P,N)=0.
\end{equation}
Now consider the maximal approximation of $N$,
\begin{equation}\tag{\ref{t2}.7}
0\longrightarrow Q\longrightarrow Y\longrightarrow N\longrightarrow0,
\end{equation}
where $Y$ is maximal Cohen-Macaulay and $Q$ has finite injective dimension.
The exact sequence (\ref{t2}.1) induces the following long exact sequence
\begin{equation}\tag{\ref{t2}.8}
\cdots\rightarrow\Tor_i^R(M,Q)\rightarrow\Tor_i^R(P,Q)\rightarrow\Tor_i^R(X,Q)\rightarrow\cdots.
\end{equation}
As we have seen in the proof of Theorem \ref{t1},
\begin{equation}\tag{\ref{t2}.9}
\Tor_i^R(M,Q)\cong\gt_i^R(M,Q)=0 \text{ for all } i>0.
\end{equation}
Since $\id_R(Q)<\infty$ and $X$ is totally reflexive,
\begin{equation}\tag{\ref{t2}.10}
\Tor_i^R(X,Q)\cong\gt_i^R(X,Q)=0 \text{ for all } i>0,
\end{equation}
by Theorem \ref{rel}(iv). It follows from (\ref{t2}.8), (\ref{t2}.9) and (\ref{t2}.10) that
\begin{equation}\tag{\ref{t2}.11}
\Tor_i^R(P,Q)=0 \text{ for all } i>0.
\end{equation}
As $Y$ is maximal Cohen-Macaulay and $P$ has finite projective dimension,
\begin{equation}\tag{\ref{t2}.12}
\Tor_i^R(P,Y)=0 \text{ for all } i>0.
\end{equation}
by \cite[Lemma 2.2]{Yo}. The exact sequence (\ref{t2}.7) induces the following long exact sequence
\begin{equation}\tag{\ref{t2}.13}
\cdots\rightarrow\Tor_i^R(P,Y)\rightarrow\Tor_i^R(P,N)\rightarrow\Tor_{i-1}^R(P,Q)\rightarrow\cdots.
\end{equation}
It follows from (\ref{t2}.6), (\ref{t2}.11), (\ref{t2}.12) and (\ref{t2}.13) that
$\Tor_i^R(P,N)=0$ for all $i>0$. Hence, by Lemma
\ref{l1}, and the Auslander-Buchsbaum formula,
\[\begin{array}{rl} \tag{\ref{t2}.14}
\depth_R(P\otimes_RN)=\depth_R(N)-\pd_R(P)\\
=\depth_R(N)-\gd_R(M)<\depth_R(N).
\end{array}\]
It follows from (\ref{t2}.3), (\ref{t2},14) and the exact sequence
(\ref{t2},5) that
$$\depth_R(M\otimes_RN)=\depth_R(P\otimes_RN)=\depth_R(N)-\gd_R(M).$$ Therefore the assertion follows from the Auslander-Bridger formula.
\end{proof}
\begin{thm}\label{t4}
Let $R$ be a Cohen-Macaulay local ring of dimension $d$ and let $M$, $N$ be $R$--modules. Assume that $M$ is totally reflexive. If $\Tt_i^R(M,N)=0$ for $-d\leq i\leq 0$, then
$\depth_R(M\otimes_RN)+\depth R=\depth_R(M)+\depth_R(N)$.
\end{thm}
\begin{proof}
Without loss of generality, we may assume that $R$ is Cohen-Macaulay ring with canonical module.
Consider the following exact sequence:
\begin{equation}\tag{\ref{t4}.1}
0\longrightarrow N\longrightarrow Q\longrightarrow X\longrightarrow0,
\end{equation}
where $\id_R(Q)<\infty$ and $X$ is maximal Cohen-Macaulay (see \cite{AuBu}).
The exact sequence (\ref{t4}.1) induces the following long exact sequence
\begin{equation}\tag{\ref{t4}.2}
\cdots\rightarrow\Tt_{i+1}^R(M,X)\rightarrow\Tt_i^R(M,N)\rightarrow\Tt_i^R(M,Q)\rightarrow\cdots,
\end{equation}
of Tate homology modules by \cite[Proposition 2.8]{CJ1}. As $\id_R(Q)<\infty$, $\Tt_i^R(M,Q)=0$ for all $i\in\ZZ$ by Theorem \ref{tate}(iv). Therefore, $\Tt_{i+1}^R(M,X)\cong\Tt_i^R(M,N)$ for all $i\in\ZZ$ by (\ref{t4}.2).
In particular, by our assumption
\begin{equation}\tag{\ref{t4}.3}
\Tt_i^R(M,X)=0 \text{ for } -d<i\leq 1.
\end{equation}
As $M$ is totally reflexive, $\Tor_1^R(M,X)=0$ by (\ref{t}) and (\ref{t4}.3). Therefore, the exact sequence (\ref{t4}.1) induces the following
exact sequence
\begin{equation}\tag{\ref{t4}.4}
0\longrightarrow M\otimes_RN\longrightarrow M\otimes_RQ\longrightarrow M\otimes_RX\longrightarrow0.
\end{equation}
As $Q$ has finite injective dimension and $M$ is totally reflexive, $\Tor_i^R(M,Q)\cong\gt_i^R(M,Q)=0$ for all $i>0$ by Theorem \ref{rel}(iv) and Lemma \ref{l3}.
Therefore, 
\begin{equation}\tag{\ref{t4}.5}
\depth_R(M\otimes_RQ)+d=\depth_R(M)+\depth_R(Q),
\end{equation}
by Lemma \ref{l1}. Note that, by Auslander-Bridger formula, $M$ is maximal Cohen-Macaulay. Hence, it follows from (\ref{t4}.5) and the exact sequence (\ref{t4}.1) that 
\begin{equation}\tag{\ref{t4}.6}
\depth_R(M\otimes_RQ)=\depth_R(Q)=\depth_R(N).
\end{equation}
On the other hand, by (\ref{t4}.3) and Lemma \ref{l2},
\begin{equation}\tag{\ref{t4}.7}
\depth_R(M\otimes_RX)=d.
\end{equation}
It follows from (\ref{t4}.6), (\ref{t4}.7) and the exact sequence (\ref{t4}.4) that $\depth_R(M\otimes_RN)=\depth_R(N)$. In other words, the pair $(M,N)$ satisfies the depth formula.
\end{proof}
Note that every finitely generated module over a Gorenstein local ring has finite Gorenstein dimension. Therefore, Theorem \ref{MT} is a special case of the following result.
\begin{cor}\label{c1}
Let $R$ be a Cohen-Macaulay local ring and let $M$ and $N$ be
$R$--modules such that $M$ has finite Gorenstein dimension. Set
$m=\max\{\depth_R(M), \depth_R(N)\}$ and $n=\gd_R(M)$. Assume that
the following conditions hold:
\begin{enumerate}[(i)]
\item{$\Tt_i^R(M,N)=0$ for $-m\leq i\leq 0$.}
\item{If $n>0$, assume $\gt_i^R(M,N)=0$ for $1\leq i\leq n$.}
\end{enumerate}
Then $\depth_R(M\otimes_RN)+\depth R=\depth_R(M)+\depth_R(N)$.
\end{cor}
\begin{proof}
If $n=0$ then $M$ is maximal Cohen-Macaulay by Auslander-Bridger formula and so $m=\dim R$. Hence, the assertion follows from Theorem \ref{t4}.
If $n>0$ then the assertion is clear by Theorem \ref{t2}.
\end{proof}

As an application, we have the following result.
\begin{cor}\label{c2}
Let $R$ be a Cohen-Macaulay local ring and let $M$ and $N$ be
$R$--modules such that $M$ has finite Gorenstein dimension. Set $m=\max\{\depth_R(M), \depth_R(N)\}$ and $n=\gd_R(M)$. Assume that
the following conditions hold:
\begin{enumerate}[(i)]
\item{$\Tt_i^R(M,N)=0$ for $-m\leq i\leq n$.}
\item{If $n>0$, assume $\Tor_i^R(M,N)=0$ for $1\leq i\leq n$.}
\end{enumerate}
Then $\depth_R(M\otimes_RN)+\depth R=\depth_R(M)+\depth_R(N)$.
\end{cor}
\begin{proof}
If $n=0$, then $\gt_i^R(M,N)=0$ for all $i>0$ and $M$ is maximal Cohen-Macaulay by Auslander-Bridger formula. Hence, the assertion is clear by Theorem \ref{t4}.
Assume that $n>0$. It follows from (i), (ii) and Theorem \ref{rel}(ii) that $\gt_i^R(M,N)=0$ for $1\leq i\leq n$. Now the assertion is clear by Corollary \ref{c1}.
\end{proof}
We finish this section by proving a vanishing result for local rings that are not necessarily Cohen-Macaulay.

\begin{thm}\label{t3}
Let $R$ be a local ring of depth $d$ and let $M$ and $N$ be
$R$--modules of finite Gorenstein dimension. Assume that the
following conditions hold:
\begin{enumerate}[(i)]
\item{$\Tt_i^R(M,N)=0$ for $-d\leq i\leq 0$.}
\item{$\gt_i^R(M,N)=0$ for $1\leq i\leq d$.}
\end{enumerate}
Then either $\depth_R(M\otimes_RN)+\depth R=\depth_R(M)+\depth_R(N)$ or $\depth_R(M\otimes_RN)>d$.
\end{thm}
\begin{proof}
Assume that $\depth_R(M\otimes_RN)\leq d$. We must prove that $(M,N)$ satisfies the depth formula. If $M$ and $N$ are totally reflexive, then the assertion follows from Lemma \ref{l2} and Auslander-Bridger formula. Assume that either $M$ or $N$ is not totally reflexive.
Note that by Theorem \ref{tate}(ii) and Theorem \ref{rel}(iii), Tate and relative homology are balanced for modules of finite Gorenstein dimension.
Therefore, without loss of generality we may assume that $\gd_R(M)>0$. Consider the following exact sequence:
\begin{equation}\tag{\ref{t3}.1}
0\longrightarrow Q\longrightarrow Y\longrightarrow N\longrightarrow0,
\end{equation}
where $\pd_R(Q)<\infty$ and $Y$ is totally reflexive. By applying the exact sequence (\ref{t3}.1) instead of
the exact sequence (\ref{t2}.7) in the proof of Theorem
\ref{t2}, similarly one can deduce the assertion.
\end{proof}
\section*{Acknowledgments}
I am grateful to Olgur Celikbas for his comments on an earlier version of this paper and for his help in finding Example \ref{e1}.
\bibliographystyle{amsplain}

\end{document}